\documentclass[11pt,reqno]{amsart}
\UseRawInputEncoding

\usepackage{latexsym,bm, amsfonts,amsthm,mathrsfs,CJK}
\usepackage{amsmath}

\makeatletter
\newcommand{\rmnum}[1]{\romannumeral #1}
\newcommand{\Rmnum}[1]{\expandafter\@slowromancap\romannumeral #1@}
\makeatother
\newtheorem{thm}{Theorem}[section]
\newtheorem{lem}[thm]{Lemma} 
 \newtheorem{defn}[thm]{Definition} \newtheorem{rmk}[thm]{Remark}

\newcommand{\dif}{\mathrm{d}} \DeclareMathAlphabet{\mathsfsl}{OT1}{cmss}{m}{sl} \DeclareMathAlphabet{\mathpzc}{OT1}{pzc}{m}{it}

    \newcommand{\ee}{\mathbb{E}}
   \newcommand{\nn}{\mathbb{N}} \newcommand{\rr}{\mathbb{R}}
    
\newcommand{\vv}{\mathbb{V}}
 \def\CC{\mathcal C}   \def\FF{\mathcal F}  \def\HH{\mathcal H}

 \def\d"{^{\prime\prime}} \def\d'{^{\prime}}

\begin{document}\title[]{Limiting behaviour of moving average processes genenrated by negatively dependent random variables under sub-linear expectations}
	\thanks{Project supported by Doctoral Scientific Research Starting Foundation of Jingdezhen Ceramic University ( Nos.102/01003002031 ), Natural Science Foundation Program of Jiangxi Province 20202BABL211005, and National Natural Science Foundation of China (Nos. 61662037), Jiangxi Province Key S\&T Cooperation Project (Nos. 20212BDH80021).}
	\subjclass[2000]{60F15; 60F05} \keywords{Negatively dependent random variables; Complete convergence; Complete moment convergence; Sub-linear expectations}
	\date{} \maketitle
	
	\begin{center}
		Mingzhou Xu~\footnote{Email: mingzhouxu@whu.edu.cn} \quad Kun Cheng~\footnote{Email: chengkun0010@126.com}\quad Wangke Yu~\footnote{Email:  ywkyyy@163.com}
		\\
		School of Information Engineering, Jingdezhen Ceramic University\\
		Jingdezhen 333403, China
		
	\end{center}
	\begin{abstract}
		Let $\{Y_i,-\infty<i<\infty\}$ be a doubly infinite sequence of identically distributed, negatively dependent random variables under sub-linear expectations, $\{a_i,-\infty<i<\infty\}$ be an absolutely summable sequence of real numbers. In this article, we study complete convergence and Marcinkiewicz-Zygmund strog law of large numbers for the partial sums of moving average processes $\{X_n=\sum_{i=-\infty}^{\infty}a_{i}Y_{i+n},n\ge 1\}$ based on the sequence $\{Y_i,-\infty<i<\infty\}$ of identically distributed, negatively dependent random variables under sub-linear expectations, complementing the result of [Chen, et al., 2009. Limiting behaviour of moving average processes under $\varphi$-mixing assumption. Statist. Probab. Lett. 79, 105-111].
	\end{abstract}
	
	\section{Introduction }
	\cite{peng2007g,peng2010nonlinear,peng2019nonlinear} firstly introduced the important concepts of the sub-linear expectations space to study the uncertainty in probability. Inspired by the seminal works of \cite{peng2007g,peng2010nonlinear,peng2019nonlinear}, many scholars try to study the results under sub-linear expectations space, generalizing the corresponding ones in classic probability space.  \cite{zhang2015donsker,zhang2016exponential,zhang2016rosenthal} established Donsker's invariance principle, exponential inequalities and Rosenthal's inequality under sub-linear expectations. \cite{wu2020precise} obtained precise asymptotics for complete integral convergence under sub-linear expectations. Under sub-linear expectations, \cite{xu2022small} investigated how small the increments of $G$-Brownian motion are. For more limit theorems under sub-linear expectations, the interested readers could refer to \cite{xu2019three,xu2020law}, \cite{wu2018strong}, \cite{zhang2018marcinkiewicz}, \cite{zhong2017complete},  \cite{hu2017some}, \cite{chen2016strong}, \cite{chen2022complete}, \cite{zhang2016strong}, \cite{hu2014big}, \cite{gao2011large}, \cite{kuczmaszewska2020complete}, \cite{xu2021precise,xu2021equivalent,xu2021convergence,xu2022small,xu2022note} and references therein.
	
	In classic probability space, \cite{zhang2017further} studied the complete moment convergence of the partial sums of moving average processes under some proper assumptions,  \cite{chen2009limiting} proved complete convergence of moving average processes under $\varphi$-mixing assumption. For references on complete convergence and complete moment convergence in linear expectation space, the interested reader could refer to \cite{ko2015complete}, \cite{meng2021convergence},  \cite{hosseini2019complete}, \cite{meng2022complete} and refercences therein. Inspired by the work of \cite{chen2009limiting}, we try to study complete convergence and Marcinkiewicz-Zygmund strog law of large numbers for the partial sums of moving average processes generated by negatively dependent random variables under sub-linear expectations, which complements the corresponding results in \cite{meng2022complete}.
	
	We organize the rest of this paper as follows. We give necessary basic notions, concepts and relevant  properties, and cite necessary lemmas under sub-linear expectations in the next section. In Section 3, we give our main results, Theorems \ref{thm3.1}-\ref{thm3.2},  the proofs of which are  presented in Section 4.
	
	\section{Preliminaries}
	As in \cite{xu2021convergence}, we use similar notations as in the work by \cite{peng2010nonlinear} \cite{peng2019nonlinear}, \cite{chen2016strong}, \cite{zhang2016rosenthal}. Assume that $(\Omega,\FF)$ is a given measurable space. Suppose that $\HH$ is a subset of all random variables on $(\Omega,\FF)$ such that  $X_1,\cdots,X_n\in \HH$ implies $\varphi(X_1,\cdots,X_n)\in \HH$ for each $\varphi\in \CC_{l,Lip}(\rr^n)$, where $\CC_{l,Lip}(\rr^n)$ represents the linear space of (local lipschitz) function $\varphi$ fulfilling
	$$
	|\varphi(\mathbf{x})-\varphi(\mathbf{y})|\le C(1+|\mathbf{x}|^m+|\mathbf{y}|^m)(|\mathbf{x}-\mathbf{y}|), \forall \mathbf{x},\mathbf{y}\in \rr^n
	$$
	for some $C>0$, $m\in \nn$ depending on $\varphi$.
	\begin{defn}\label{defn1} A sub-linear expectation $\ee$ on $\HH$ is a functional $\ee:\HH\mapsto \bar{\rr}:=[-\infty,\infty]$ satisfying the following properties: for all $X,Y\in \HH$, we have
		\begin{description}
			\item[\rm (a)]  Monotonicity: If $X\ge Y$, then $\ee[X]\ge \ee[Y]$;
			\item[\rm (b)] Constant preserving: $\ee[c]=c$, $\forall c\in\rr$;
			\item[\rm (c)] Positive homogeneity: $\ee[\lambda X]=\lambda\ee[X]$, $\forall \lambda\ge 0$;
			\item[\rm (d)] Sub-additivity: $\ee[X+Y]\le \ee[X]+\ee[Y]$ whenever $\ee[X]+\ee[Y]$ is not of the form $\infty-\infty$ or $-\infty+\infty$.
		\end{description}
		
	\end{defn}
	A set function $V:\FF\mapsto[0,1]$ is named to be a capacity if
	\begin{description}
		\item[\rm (a)]$V(\emptyset)=0$, $V(\Omega)=1$;
		\item[\rm (b)]$V(A)\le V(B)$, $A\subset B$, $A,B\in \FF$.\\
		Moreover, if $V$ is continuous, then $V$ should obey
		\item[\rm (c)] $V(A_n)\uparrow V(A)$, if $A_n\uparrow A$.
		\item[\rm (d)] $V(A_n)\downarrow V(A)$, if $A_n\downarrow A$.
	\end{description}
	A capacity $V$ is called to be sub-additive if $V(A+B)\le V(A)+V(B)$, $A,B\in \FF$.

	In this article, given a sub-linear expectation space $(\Omega, \HH, \ee)$, set $\vv(A):=\inf\{\ee[\xi]:I_A\le \xi, \xi\in \HH\}$, $\forall A\in \FF$ (see (2.3) and the definitions of $\vv$ above (2.3) in \cite{zhang2016exponential}. $\vv$ is a sub-additive capacity.
	Define
	$$
	C_{\vv}(X):=\int_{0}^{\infty}\vv(X>x)\dif x +\int_{-\infty}^{0}(\vv(X>x)-1)\dif x.
	$$

	Suppose that $\mathbf{X}=(X_1,\cdots, X_m)$, $X_i\in\HH$ and $\mathbf{Y}=(Y_1,\cdots,Y_n)$, $Y_i\in \HH$  are two random vectors on  $(\Omega, \HH, \ee)$. $\mathbf{Y}$ is named to be negatively dependent to $\mathbf{X}$, if for each Borel-measurable function $\psi_1$ on $\rr^m$, $\psi_2$ on $\rr^n$, we have $\ee[\psi_1(\mathbf{X})\psi_2(\mathbf{Y})]\le\ee[\psi_1(\mathbf{X})] \ee[\psi_2(\mathbf{Y})]$ whenever $\psi_1(\mathbf{X})\ge 0$, $\ee[\psi_2(\mathbf{Y})]\ge 0 $, $\ee[\psi_1(\mathbf{X})\psi_2(\mathbf{Y})]<\infty$, $\ee[|\psi_1(\mathbf{X})|]<\infty$, $\ee[|\psi_2(\mathbf{Y})|]<\infty$, and either $\psi_1$ and $\psi_2$ are coordinatewise nondecreasing or $\psi_1$ and $\psi_2$ are coordinatewise nonincreasing (see Definition 2.3 of \cite{zhang2016exponential}, Definition 1.5 of \cite{zhang2016rosenthal}, Definition 2.5 in \cite{chen2016strong}).
	$\{X_n\}_{n=1}^{\infty}$ is called a sequence of negatively dependent random variables, if $X_{n+1}$ is negatively dependent to $(X_1,\cdots,X_n)$ for each $n\ge 1$.
	
	Assume that $\mathbf{X}_1$ and $\mathbf{X}_2$ are two $n$-dimensional random vectors defined, respectively, in sub-linear expectation spaces $(\Omega_1,\HH_1,\ee_1)$ and $(\Omega_2,\HH_2,\ee_2)$. They are named identically distributed if  for every Borel-measurable function $\psi$ such that $\psi(\mathbf{X}_1)\in \HH_1, \psi(\mathbf{X}_2)\in \HH_2$,
	$$
	\ee_1[\psi(\mathbf{X}_1)]=\ee_2[\psi(\mathbf{X}_2)], \mbox{  }
	$$
	whenever the sub-linear expectations are finite. $\{X_n\}_{n=1}^{\infty}$ is called to be identically distributed if for each $i\ge 1$, $X_i$ and $X_1$ are identically distributed.
	
	In this article we suppose that $\ee$ is countably sub-additive, i.e., $\ee(X)\le \sum_{n=1}^{\infty}\ee(X_n)$, whenever $X\le \sum_{n=1}^{\infty}X_n$, $X,X_n\in \HH$, and $X\ge 0$, $X_n\ge 0$, $n=1,2,\ldots$. Write $S_n=\sum_{i=1}^{n}X_i$, $n\ge 1$. Let $C$ represent a positive constant which may differ from place to place. $I(A)$ or $I_A$ stand for the indicator function of $A$.
	
	As discussed in \cite{zhang2016rosenthal}, by the definition of negative dependence, if $X_1, X_2,\ldots, X_n$ are negatively dependent random variables and $f_1$, $f_2, \ldots, f_n$ are all non increasing ( or non decreasing) functions, then $f_1(X_1)$, $f_2(X_2), \ldots, f_n(X_n)$ are still negatively dependent random variables.

	We cite the following inequalities under sub-linear expectations.
	\begin{lem}\label{lem01}(See Lemma 4.5 (\rmnum{3}) of \cite{zhang2016exponential}) If $\ee$ is countably sub-additive under sub-linear expectation space $(\Omega, \HH, \ee)$, then for $X\in \HH$,
		$$
		\ee|X|\le C_{\vv}\left(|X|\right).
		$$
	\end{lem}
	\begin{lem}\label{lem02}(See Theorem 2.1 of \cite{zhang2016rosenthal}) Assume that $p>1$ and $\{Y_n;n\ge 1\}$ is a sequence of negatively dependent random varables under sub-linear expectation space $(\Omega,\HH,\ee)$. Then for each $n\ge 1$, there exists a positive constant $C=C(p)$ depending on $p$ such that for $p\ge2$,
		\begin{eqnarray}\label{01}
			\nonumber &&\ee\left[\max_{0\le i\le n}\left|\sum_{j=1}^{i}Y_j\right|^p\right]\\
			&&\quad\le C \left\{\sum_{i=1}^{n}\ee\left|Y_i\right|^p+\left(\sum_{i=1}^{n}\ee Y_i^2\right)^{p/2}+\left(\sum_{i=1}^{n}\left[(-\ee(-Y_k))^{-}+(\ee(Y_k))^{+}\right]\right)^p\right\}.
		\end{eqnarray}
		
	\end{lem}

	\section{Main Results}
	Our main results are the following.
	\begin{thm}\label{thm3.1} Assume that $h$ is a function slowly varying at infinity, $1\le p<2$, and $r>1$. Let $\{X_n,  n\ge 1\}$ be a moving average process based on a sequence of $\{Y_i,-\infty<i<\infty\}$ of negatively dependent random variables, identically distributed as $Y$ under sub-linear expectation space $(\Omega,\HH,\ee)$. Suppose that $\ee(Y)=\ee(-Y)=0$ and $C_{\vv}\left(|Y|^{rp}h(|Y|^p)\right)<\infty$. Then \\
		{\rm (\rmnum{1})}  $\sum_{n=1}^{\infty}n^{r-2}h(n)\vv\left\{\max_{1\le k\le n}|S_k|\ge \varepsilon n^{1/p}\right\}<\infty$, for all $\varepsilon>0$,\\
		and\\
		{\rm (\rmnum{2}}   $\sum_{n=1}^{\infty}n^{r-2}h(n)\vv\left\{\sup_{k\ge n}\left|S_k/k^{1/p}\right|\ge \varepsilon\right\}<\infty$, for all $\varepsilon>0$.
	\end{thm}
	The following theorem investigates the case $r=1$.
	
	\begin{thm}\label{thm3.2}  Assume that $h$ is a function slowly varying at infinity and $1\le p<2$. Suppose that $\sum_{i=-\infty}^{\infty}|a_i|^{\theta}<\infty$, where $\theta\in (0, 1)$ if $p=1$ and $\theta=1$ if $1<p<2$. Assume that $\{X_n,  n\ge 1\}$ is a moving average process based on a sequence of $\{Y_i,-\infty<i<\infty\}$ of negatively dependent random variables, identically distributed as $Y$ under sub-linear expectation space $(\Omega,\HH,\ee)$. Suppose that $\ee(Y)=\ee(-Y)=0$ and $C_{\vv}\left(|Y|^ph(|Y|^p)<\infty\right)$. Then
		$$
		\sum_{n=1}^{\infty}\frac{h(n)}{n}\vv\left\{\max_{1\le k\le n}|S_k|\ge \varepsilon n^{1/p}\right\}<\infty, \mbox{   for all $\varepsilon>0$.}
		$$
		In particular, the conditions that  $\ee Y=\ee(-Y)=0$, $C_{\vv}\left(|Y|^p\right)<\infty$ and $\vv$ is continuous imply the following Marcinkiewicz-Zygmund strong law of large numbers,
		$S_n/n^{1/p}\rightarrow 0$ a.s. $\vv$, i.e.,
		\[
		\vv\left\{\Omega\setminus\{\lim_{n\rightarrow\infty}S_n/n^{1/p}=0\}\right\}=0.
		\]
	\end{thm}
	\begin{rmk}Theorem \ref{thm3.2} complements Theorem 1 for independent, identically distributed random variables under sub-linear expectations in \cite{zhang2018marcinkiewicz}.
	\end{rmk}

	\section{Proofs of Main Results}
	We first present some lemmas.
	\begin{lem}\label{lem03} Suppose $r>1$, and $1\le p<2$. Then for any $\varepsilon>0$,
		\[\sum_{n=1}^{\infty}n^{r-2}h(n)\vv\left\{\sup_{k\ge n}\left|S_k/k^{1/p}\right|\ge \varepsilon\right\}\le \sum_{n=1}^{\infty}n^{r-2}h(n)\vv\left\{\sup_{k\ge n}\left|S_k\right|\ge \left(\varepsilon/2^{2/p}\right)n^{1/p}\right\}.
		\]
	\end{lem}
	\begin{proof}
		We see that
		$$
		\begin{aligned}
			\sum_{n=1}^{\infty}n^{r-2}h(n)\vv\left\{\sup_{k\ge n}\left|S_k/k^{1/p}\right|\ge \varepsilon\right\}&=\sum_{m=1}^{\infty}\sum_{n=2^{m-1}}^{2^{m}-1}n^{r-2}h(n)\vv\left\{\sup_{k\ge n}\left|S_k/k^{1/p}\right|\ge \varepsilon\right\}\\
			&\le C\sum_{m=1}^{\infty}\vv\left\{\sup_{k\ge 2^{m-1}}\left|S_k/k^{1/p}\right|\ge \varepsilon\right\}\sum_{n=2^{m-1}}^{2^{m}-1}2^{m(r-2)}h(2^m)\\
			&\le C\sum_{m=1}^{\infty}2^{m(r-1)}h(2^m)\vv\left\{\sup_{k\ge 2^{m-1}}\left|S_k/k^{1/p}\right|\ge \varepsilon\right\}\\
			&=C\sum_{m=1}^{\infty}2^{m(r-1)}h(2^m)\vv\left\{\sup_{l\ge m}\max_{2^{l-1}\le k< 2^{l}}\left|S_k\right|\ge \varepsilon 2^{(l-1)/p}\right\}\\
			&\le C\sum_{m=1}^{\infty}2^{m(r-1)}h(2^m)\sum_{l=m}^{\infty}\vv\left\{\max_{1\le k< 2^{l}}\left|S_k\right|\ge \varepsilon 2^{(l-1)/p}\right\}\\
			&=C\sum_{l=1}^{\infty}\vv\left\{\max_{1\le k< 2^{l}}\left|S_k\right|\ge \varepsilon 2^{(l-1)/p}\right\}\sum_{m=1}^{l}2^{m(r-1)}h(2^m)\\
			&\le C\sum_{l=1}^{\infty}2^{l(r-1)}h(2^l)\vv\left\{\max_{1\le k< 2^{l}}\left|S_k\right|\ge \varepsilon 2^{(l-1)/p}\right\}\\
			&\le C\sum_{l=1}^{\infty}\sum_{n=2^{l}}^{2^{l+1}-1}n^{r-2}h(n)\vv\left\{\max_{1\le k\le n}\left|S_k\right|\ge (\varepsilon/2^{2/p})n^{1/p}\right\}\\
			&\le C\sum_{n=1}^{\infty}n^{r-2}h(n)\vv\left\{\max_{1\le k\le n}\left|S_k\right|\ge (\varepsilon/2^{2/p})n^{1/p}\right\}.
		\end{aligned}
		$$
	\end{proof}
	\begin{lem}\label{lem04}Let $Y$ be a random variable with $C_{\vv}\left(|Y|^{rp}h(|Y|^p)\right)<\infty$, where $r\ge 1$, and $p\ge 1$. Write $Y'=-n^{-1/p}I\{Y<-n^{-1/p}\}+YI\{|Y|\le n^{1/p}\}+n^{1/p}I\{Y>n^{1/p}\}$. If $q>rp$, then
		\begin{equation*}
			\sum_{n=1}^{\infty}n^{r-1-q/p}h(n)\ee|Y'|^q\le CC_{\vv}\left(|Y|^{rp}h(|Y|^p)\right).
		\end{equation*}
	\end{lem}
	\begin{proof}Since $r-q/p<0$, by Lemma \ref{lem01}, and similar proof of Lemma 2.2 in \cite{zhong2017complete}, we see that
		$$
		\begin{aligned}
			&\sum_{n=1}^{\infty}n^{r-1-q/p}h(n)\ee|Y'|^q\le \sum_{n=1}^{\infty}n^{r-1-q/p}h(n)C_{\vv}\left\{|Y'|^q\right\}\\
			&\quad\le \sum_{n=1}^{\infty}n^{r-1-q/p}h(n)\int_{0}^{n^{1/p}}\vv\left\{|Y'|^q>x^q\right\}qx^{q-1}\dif x\\
			&\quad\le C\int_{1}^{\infty}y^{r-1-q/p}h(y)\left[\int_{0}^{1}+\int_{1}^{y^{1/p}}\right]\vv\left\{|Y'|^q>x^q\right\}x^{q-1}\dif x\dif y\\
			&\quad\le C \int_{0}^{1}\vv\left\{|Y|^q>x\right\}\dif x\int_{1}^{\infty}y^{r-1-q/p}h(y)\dif y\\
		\end{aligned}
		$$
		$$
		\begin{aligned}
			&\quad \quad +C\int_{1}^{\infty}\vv\left\{|Y|>x\right\}x^{q-1}\int_{x^p}^{\infty}y^{r-1-q/p}h(y)\dif y\dif x\\
			&\quad\le C+C\int_{1}^{\infty}\vv\left\{|Y|>x\right\}h(x^p)x^{rp-1}\dif x\\
			&\quad\le CC_{\vv}\left(|Y|^{rp}h(|Y|^p)\right)<\infty.
		\end{aligned}
		$$
		
	\end{proof}
	
	In the rest of this paper, let $\frac12<\mu<1$, $g(y)\in \CC_{l,Lip}(\rr)$, such that $0\le g(y)\le 1$ for all $y$ and $g(y)=1$ if $|y|\le \mu$, $g(y)=0$, if $|y|>1$. And $g(y)$ is a decreasing function for $y\ge 0$. The next lemma presents an important fact in the proofs of Theorems \ref{thm3.1} and \ref{thm3.2}.
	\begin{lem}\label{lem05}
		Assume that $h$ is a function slowly varying at infinity and $p\ge 1$. Assume that $\{X_n,  n\ge 1\}$ is a moving average process based on a sequence of $\{Y_i,-\infty<i<\infty\}$ of negatively dependent random variables, identically distributed as $Y$ with $\ee(Y)=\ee(-Y)=0$, $C_{\vv}\left(|X|^p\right)<\infty$ under sub-linear expectation space $(\Omega,\HH,\ee)$. For any $\varepsilon>0$,  write
		\[
		\Rmnum{1}:=\sum_{n=1}^{\infty}n^{r-2}h(n)\vv\left\{\max_{1\le k\le n}\left|\sum_{i=-\infty}^{\infty}a_i\sum_{j=i+1}^{i+k}Y_j''\right|\ge \varepsilon n^{1/p}/2\right\},
		\]
		and
		\[
		\Rmnum{2}:=\sum_{n=1}^{\infty}n^{r-2}h(n)\vv\left\{\max_{1\le k\le n}\left|\sum_{i=-\infty}^{\infty}a_i\sum_{j=i+1}^{i+k}(Y_j'-\ee[Y_j'])\right|\ge \varepsilon n^{1/p}/4\right\},
		\]
		where
		\[Y_j'=-n^{1/p}I\{Y_j<-n^{-1/p}\}+|Y_j|I\{|Y_j|\le n^{1/p}\}+n^{1/p}I\{Y_j>n^{1/p}\},
		\]
		\[
		Y_j''=Y_j-Y_j'=(Y_j+n^{1/p})I\{Y_j<-n^{1/p}\}+(Y_j-n^{1/p})I\{Y_j>n^{1/p}\}\].
		If $\Rmnum{1}<\infty$ and $\Rmnum{2}<\infty$, then
		\[
		\sum_{n=1}^{\infty}n^{r-2}h(n)\vv\left\{\max_{1\le k\le n}\left|S_k\right|\ge \varepsilon n^{1/p}\right\}\le \Rmnum{1}+\Rmnum{2}<\infty.
		\]
	\end{lem}
	\begin{proof} Observe that
		\[
		\sum_{k=1}^{n}X_k=\sum_{k=1}^{n}\sum_{i=-\infty}^{\infty}a_iY_{i+k}=\sum_{i=-\infty}^{\infty}a_i\sum_{j=i+1}^{i+n}Y_j.
		\]
		By $\sum_{i=-\infty}^{\infty}|a_i|<\infty$, $\ee(Y_j)=\ee(-Y_j)=0$, and Proposition 1.3.7 of \cite{peng2019nonlinear}, Lemma \ref{lem01}, we have
		$$
		\begin{aligned}
			&n^{-1/p}\left|\sum_{i=-\infty}^{\infty}a_i\sum_{j=i+1}^{i+n}\ee Y_j'\right|=n^{-1/p}\left|\sum_{i=-\infty}^{\infty}a_i\sum_{j=i+1}^{i+n}\ee\left[Y_j'-Y_j\right] \right|\\
			&\quad \le n^{-1/p}\sum_{i=-\infty}^{\infty}|a_i|\sum_{j=i+1}^{i+n}\ee |Y_j-Y_j'| \le C n^{-1/p}\ee|Y_1''| \le C n^{-1/p}\ee (n^{-1/p})^{p-1}|Y_1''|^p\\
			&\quad \le C n^{1-1/p}\ee|Y_1|^p\left(1-g\left(\frac{|Y_1|}{n^{1/p}}\right)\right)\le CC_{\vv}\left\{|Y_1|^p\left(1-g\left(\frac{|Y_1|}{n^{1/p}}\right)\right)\right\}\\
			&\quad \le CC_{\vv}\left\{|Y_1|^pI\{|Y_1|\ge \mu n^{1/p}\}\right\}\rightarrow 0, \mbox{   $n\rightarrow 0$.}
		\end{aligned}
		$$
		Therefore for $n$ sufficiently large, we see that
		\[
		n^{-1/p}\left|\sum_{i=-\infty}^{\infty}a_i\sum_{j=i+1}^{i+n}\ee Y_j'\right|<\varepsilon/4.
		\]
		Then
		$$
		\begin{aligned}
			&\sum_{n=1}^{\infty}n^{r-2}h(n)\vv\left\{\max_{1\le k\le n}\left|S_k\right|\ge \varepsilon n^{1/p}\right\}\\
			&\quad\le C\sum_{n=1}^{\infty}n^{r-2}h(n)\vv\left\{\max_{1\le k\le n}\left|\sum_{i=-\infty}^{\infty}a_i\sum_{j=i+1}^{i+k}Y_j''\right|\ge \varepsilon n^{1/p}/2\right\}\\
			&\quad \quad +\sum_{n=1}^{\infty}n^{r-2}h(n)\vv\left\{\max_{1\le k\le n}\left|\sum_{i=-\infty}^{\infty}a_i\sum_{j=i+1}^{i+k}(Y_j'-\ee[Y_j'])\right|\ge \varepsilon n^{1/p}/4\right\}\\
			&\quad=:\Rmnum{1}+\Rmnum{2}.
		\end{aligned}
		$$

	\end{proof}
	
	\begin{proof}[\emph{Proof of Theorem \ref{thm3.1}}] According to Lemma \ref{lem03}, it is enough to establish that $(\rm{\rmnum{1}})$ holds. By Lemma \ref{lem05}, we only need to establish that $\Rmnum{1}<\infty$ and $\Rmnum{2}<\infty$.
		
		For $\Rmnum{1}$, by Markov inequality under sub-linear expectations, Lemma \ref{lem01}, and similar proof of Lemma 2.2 in \cite{zhong2017complete}, we obtain
		$$
		\begin{aligned}
			\Rmnum{1}&\le C\sum_{n=1}^{\infty}n^{r-2}h(n)n^{-1/p}\ee\max_{1\le k\le n}\left|\sum_{i=-\infty}^{\infty}a_i\sum_{j=i+1}^{i+k}Y_j''\right|\\
			&\le C\sum_{n=1}^{\infty}n^{r-1-1/p}h(n)\ee|Y_1''|\\
			&\le C\sum_{n=1}^{\infty}n^{r-1-1/p}h(n)C_{\vv}\left\{|Y_1''|\right\}\\
			&\le C\sum_{n=1}^{\infty}n^{r-1-1/p}h(n)\int_{0}^{\infty}\vv\left\{|Y_1''|>x\right\}\dif x\\
			&\le C\sum_{n=1}^{\infty}n^{r-1-1/p}h(n)\left[\vv\left\{|Y|>n^{1/p}\right\}n^{1/p}+\int_{n^{1/p}}^{\infty}\vv\left\{|Y|>x\right\}\dif x\right]\\
			&\le C\int_{1}^{\infty}x^{r-1}h(x)\vv\left\{|Y|>x^{1/p}\right\}\dif x+C\int_{1}^{\infty}y^{r-1-1/p}h(y)\int_{y^{1/p}}^{\infty}\vv\left\{|Y|>x\right\}\dif x\dif y\\
			&\le C\int_{1}^{\infty}\vv\left\{|Y|^{pr}h(|Y|^p)>x^rh(x)\right\}\dif (x^rh(x))+C\int_{1}^{\infty}\vv\left\{|Y|>x\right\}\dif x\int_{1}^{x^p}y^{r-1-1/p}h(y)\dif y\\
			&\le CC_{\vv}\left\{|Y|^{pr}h(|Y|^p)\right\}+C\int_{1}^{\infty}\vv\left\{|Y|>x\right\}x^{rp-1}h(x^p)\dif x\\
			&\le CC_{\vv}\left\{|Y|^{pr}h(|Y|^p)\right\}<\infty.
		\end{aligned}
		$$
		For $\Rmnum{2}$, by Markov inequality under sub-linear expectations, H\"{o}lder inequality, Lemma \ref{lem02}, we see that for any $q>2$,
		$$
		\begin{aligned}
			\Rmnum{2}&\le C\sum_{n=1}^{\infty}n^{r-2}h(n)n^{-q/p}\ee\max_{1\le k\le n}\left|\sum_{i=-\infty}^{\infty}a_i\sum_{j=i+1}^{i+k}(Y_j'-\ee[Y_j'])\right|^q\\
			&\le  C\sum_{n=1}^{\infty}n^{r-2}h(n)n^{-q/p}\ee\left[\sum_{i=-\infty}^{\infty}|a_i|^{1-1/q}\left(|a_i|^{1/q}\max_{1\le k\le n}\left|\sum_{j=i+1}^{i+k}(Y_j'-\ee[Y_j'])\right|\right)\right]^q\\
			&\le C\sum_{n=1}^{\infty}n^{r-2-q/p}h(n)\left(\sum_{i=-\infty}^{\infty}|a_i|\right)^{q-1}\sum_{i=-\infty}^{\infty}|a_i|\ee\max_{1\le k\le n}\left|\sum_{j=i+1}^{i+k}(Y_j'-\ee[Y_j'])\right|^q\\
			&\le C\sum_{n=1}^{\infty}n^{r-2-q/p}h(n)\left[n\left(\left|\ee[-Y_1']\right|+\left|\ee[Y_1']\right|\right)\right]^q\\
			&\quad + C\sum_{n=1}^{\infty}n^{r-2-q/p}h(n)\left(n\ee|Y_1'|^2\right)^{q/2}\\
			&\quad +C\sum_{n=1}^{\infty}n^{r-1-q/p}h(n)\ee|Y_1'|^q\\
			&=:\Rmnum{2}_1+\Rmnum{2}_2+\Rmnum{2}_3.
		\end{aligned}
		$$
		To establish $\Rmnum{2}_1<\infty$, in view of $\ee(-Y)=\ee(Y)=0$, and $|\ee(X)-\ee(Y)|\le \ee|X-Y|$, by Lemma \ref{lem01}, we see that
		$$
		\begin{aligned}
			&\left[n\left(\left|\ee[-Y_1']\right|+\left|\ee[Y_1']\right|\right)\right]^q\\
			&\quad \le \left[n\left(\left|\ee[-Y_1']-\ee[-Y_1]\right|+\left|\ee[Y_1']-\ee[Y_1]\right|\right)\right]^q\\
			&\quad \le C n^{q}\left(\ee|Y_1''|\right)^q\le Cn^q\left(\ee|Y_1|\left(1-g\left(\frac{|Y_1|}{n^{1/p}}\right)\right)\right)^q\\
			&\quad \le Cn^q \left(C_{\vv}\left\{|Y_1|\left(1-g\left(\frac{|Y_1|}{n^{1/p}}\right)\right)\right\}\right)^q\\
			&\quad \le Cn^q \left(C_{\vv}\left\{|Y_1|I\{|Y_1|>\mu n^{1/p}\}\right\}\right)^q\\
			&\quad \le Cn^q \left(\int_{0}^{\infty}\vv\left\{|Y^{rp}|h(|Y|^p)I\{|Y|>\mu n^{1/p}\}>x\mu n^{(rp-1)/p}h(n)\right\}\dif x\right)^q\\
			&\quad \le Cn^q\left(\frac{C_{\vv}\{|Y|^{rp}h(|Y|^p)\}}{n^{(rp-1)/p}h(n)}\right)^q \ll Cn^{q-qr+q/p}/h(n)^q.
		\end{aligned}
		$$
		Hence
		$$
		\begin{aligned}
			\Rmnum{2}_1&\le C\sum_{n=1}^{\infty}n^{r-2-q/p}h(n)n^{q-qr+q/p}/h(n)^q\\
			&\le C\sum_{n=1}^{\infty}n^{r-2-q(r-1)}h(n)^{1-q}<\infty.
		\end{aligned}
		$$
		To prove $\Rmnum{2}_2<\infty$, we study two cases. If $rp<2$, take $q>2$, observe that in this case $r-2+q/2-rq/2<-1$. By Lemma \ref{lem01}, we see that
		$$
		\begin{aligned}
			\Rmnum{2}_2&=C\sum_{n=1}^{\infty}n^{r-2-q/p}h(n)n^{q/2}\left(\ee|Y_1'|^2\right)^{q/2}\\
			&\le C\sum_{n=1}^{\infty}n^{r-2-q/p+q/2}h(n)\left(\ee|Y_1'|^{rp}|Y_1'|^{2-rp}\right)^{q/2}\\
			&\le C\sum_{n=1}^{\infty}n^{r-2-q/p+q/2}h(n) \left(C_{\vv}\left(|Y|^{rp}\right)\right)^{q/2}n^{\frac{2-rp}{p}\frac{q}{2}}\\
			&\le  C\sum_{n=1}^{\infty}n^{r-2+q/2-rq/2}h(n)<\infty.
		\end{aligned}
		$$
		If $rp\ge 2$, take $q>pr$. Note in this case $\ee|Y|^2<C_{\vv}(|Y|^2)<\infty$. We see that
		$$
		\begin{aligned}
			\Rmnum{2}_2&=C\sum_{n=1}^{\infty}n^{r-1-q/p}h(n)\left(\ee|Y_1'|^2\right)^{q/2}\\
			&\le C\sum_{n=1}^{\infty}n^{r-1-q/p}h(n)<\infty.
		\end{aligned}
		$$
		By Lemma \ref{lem04}, we conclude that $\Rmnum{2}_3<\infty$. The proof of Theorem \ref{thm3.1} is complete.
	\end{proof}
	
	\begin{proof}[\emph{ Proof of Theorem \ref{thm3.2}}]
		By Lemma \ref{lem05}, we only need to establish that $\Rmnum{1}<\infty$ and $\Rmnum{2}<\infty$ with $r=1$. For $\Rmnum{1}$, by Markov inequality under sub-linear expectations, $C_r$ inequality, Lemma \ref{lem01}, and the proof of Lemma 2.2 of \cite{zhong2017complete} ( observe that $\theta<1$),  we see that
		$$
		\begin{aligned}
			\Rmnum{1}&\le \sum_{n=1}^{\infty}n^{-1}h(n)n^{-\theta/p}\ee\max_{1\le k\le n}\left|\sum_{i=-\infty}^{\infty}a_i\sum_{j=i+1}^{i+k}Y_j''\right|^{\theta}\\
			&\le C\sum_{n=1}^{\infty}h(n)n^{-\theta/p}\ee|Y_1''|^{\theta} \\
			&\le C\sum_{n=1}^{\infty}h(n)n^{-\theta/p}C_{\vv}\left(|Y_1''|^{\theta}\right)\\
			&\le C\sum_{n=1}^{\infty}h(n)n^{-\theta/p}C_{\vv}\left\{|Y|^{\theta}I\{|Y|> n^{1/p}\}\right\}\\
			&\le C\sum_{n=1}^{\infty}n^{-\theta/p}h(n)\int_{0}^{\infty}\vv\left\{|Y|^{\theta}I\{|Y|> n^{1/p}\}>x\right\}\dif x\\
			&\le C\int_{1}^{\infty}y^{-\theta/p}h(y)\int_{0}^{\infty}\vv\left\{|Y|^{\theta}I\{|Y|> y^{1/p}\}>x\right\}\dif x\dif y\\
			&\le C\int_{1}^{\infty}y^{-\theta/p}h(y)\left[\int_{0}^{y^{\theta/p}}+\int_{y^{\theta/p}}^{\infty}\right]\vv\left\{|Y|^{\theta}I\{|Y|>y^{1/p}\}>x\right\}\dif x\dif y
		\end{aligned}
		$$
		$$
		\begin{aligned}
			&\le C\int_{1}^{\infty}\vv\left\{|Y|>y^{1/p}\right\}h(y)\dif y\\
			&\quad +C\int_{1}^{\infty}\vv\left\{|Y|^{\theta}>x\right\}\int_{1}^{x^{p/\theta}}y^{-\theta/p}h(y)\dif y\dif x\\
			&\le CC_{\vv}\left(|Y|^ph(|Y|^p)\right)+C\int_{1}^{\infty}\vv\left\{|Y|^{\theta}>x\right\}x^{p/\theta-1}h(x^{p/\theta})\dif x\\
			&\le CC_{\vv}\left(|Y|^ph(|Y|^p)\right)<\infty.
		\end{aligned}
		$$
		For $\Rmnum{2}$, by Markov inequality under sub-linear expectations, H\"{o}lder inequality, and Lemma \ref{lem01}, we see that
		$$
		\begin{aligned}
			\Rmnum{2}&\le C\sum_{n=1}^{\infty}n^{-1}h(n)n^{-2/p}\ee\max_{1\le k\le n}\left|\sum_{i=-\infty}^{\infty}a_i\sum_{j=i+1}^{i+k}(Y_j'-\ee[Y_j'])\right|^2\\
			&\le C\sum_{n=1}^{\infty}n^{-1}h(n)n^{-2/p}\ee \left(\sum_{i=-\infty}^{\infty}|a_i|^{1/2}\left(|a_i|^{1/2}\max_{1\le k\le n}\left|\sum_{j=i+1}^{i+k}(Y_j'-\ee[Y_j'])\right|\right)\right)^2\\
			&\le  C\sum_{n=1}^{\infty}n^{-1-2/p}h(n)\sum_{i=-\infty}^{\infty}|a_i|\sum_{i=-\infty}^{\infty}|a_i|\ee\max_{1\le k\le n}\left|\sum_{j=i+1}^{i+k}(Y_j'-\ee[Y_j'])\right|^2\\
			&\le C\sum_{n=1}^{\infty}n^{-1-2/p}h(n)\left[n\ee[|Y_1'|^2]+\left(\sum_{j=1}^{n}|\ee(-Y_j')|+|\ee(Y_j')|\right)^2\right]\\
			&=C\sum_{n=1}^{\infty}n^{-2/p}h(n)\ee[|Y_1'|^2] + C\sum_{n=1}^{\infty}n^{-1-2/p}h(n)\left(\sum_{j=1}^{n}|\ee(-Y_j')|+|\ee(Y_j')|\right)^2\\
			&=:\Rmnum{2}_1+\Rmnum{2}_2.
		\end{aligned}
		$$
		By Lemma \ref{lem04}, we conclude that $\Rmnum{2}_1<\infty$. By $\ee(Y_j)=\ee(-Y_j)=0$, $|\ee(X)-\ee(Y)|\le \ee|X-Y|$, $C_r$ inequality, and Lemma \ref{lem01}, we obtain
		$$
		\begin{aligned}
			\Rmnum{2}_2&\le C\sum_{n=1}^{\infty}n^{-1-2/p}h(n)\left[\sum_{j=1}^{n}\left|\ee(-Y_j')\right|+\left|\ee(Y_j')\right|\right]^2\\
			&\le C\sum_{n=1}^{\infty}n^{-1-2/p}h(n)\left[\sum_{j=1}^{n}\ee[|Y_j-Y_j'|]\right]^2\\
			&\le C\sum_{n=1}^{\infty}n^{-1-2/p}h(n)\left[\sum_{i=1}^{n}C_{\vv}\left\{|Y_i''|\right\}\right]^2\\
			&\le C\sum_{n=1}^{\infty}n^{1-2/p}h(n)\left[\int_{0}^{n^{1/p}}\vv\left\{|Y_1|>n^{1/p}\right\}\dif y+\int_{n^{1/p}}^{\infty}\vv\left\{|Y_1|>y\right\}\dif y\right]^2\\
			&\le C\sum_{n=1}^{\infty}n^{1}h(n)\left[\vv\left\{|Y|>n^{1/p}\right\}\right]^2+C\sum_{n=1}^{\infty}n^{1-2/p}h(n)\left[\int_{n^{1/p}}^{\infty}\vv\left\{|Y|>y\right\}\dif y\right]^2
		\end{aligned}
		$$
		$$
		\begin{aligned}
			&\le C\int_{1}^{\infty}xh(x)\vv^2\left\{|Y|>x^{1/p}\right\}\dif x\\
			&\quad +C\int_{1}^{\infty}x^{1-2/p}h(x)\dif x\int_{x^{1/p}}^{\infty}\vv\left\{|Y|>y\right\}\dif y\int_{x^{1/p}}^{y}\vv\left\{|Y|>z\right\}\dif z\\
			&\le C\int_{1}^{\infty}\left(x\vv\left\{|Y|^ph(|Y|^p)>xh(x)\right\}\right)\vv\left\{|Y|^ph(|Y|^p)>xh(x)\right\} \dif\left[xh(x)\right]\\
			&\quad +C\int_{1}^{\infty}\vv\left\{|Y|>y\right\}\dif y\int_{1}^{y}\vv\left\{|Y|>z\right\}\dif z\int_{1}^{z^p}x^{1-2/p}h(x)\dif x\\
			&\le C\int_{1}^{\infty}\vv\left\{|Y|^ph(|Y|^p)>xh(x)\right\} \dif\left[xh(x)\right]\\
			&\quad +C\int_{1}^{\infty}\vv\left\{|Y|>y\right\}\dif y\int_{1}^{y}\vv\left\{|Y|>z\right\}z^{2p-2}h(z^p)\dif z\\
			&\le CC_{\vv}\left(|Y|^ph(|Y|^p)\right)+C\int_{1}^{\infty}\vv\left\{|Y|>y\right\}\dif y\int_{1}^{y}\frac{\ee|Y|^p}{z^p}z^{2p-2}h(z^p)\dif z\\
			&\le CC_{\vv}\left(|Y|^ph(|Y|^p)\right)+C\int_{1}^{\infty}\vv\left\{|Y|>y\right\}C_{\vv}(|Y|^p)y^{p-1}h(y^p)\dif y\\
			&\le CC_{\vv}\left(|Y|^ph(|Y|^p)\right)+C\int_{1}^{\infty}\vv\left\{|Y|^ph(|Y|^p)>y^ph(y^p)\right\}\dif [y^ph(y^p)]\\
			&\le CC_{\vv}\left(|Y|^ph(|Y|^p)\right)<\infty.
		\end{aligned}
		$$
		Now we will establish almost sure convergence under $\vv$. By $\ee(Y_1)=\ee(-Y_1)=0$ and $C_{\vv}\left(|Y|^p\right)<\infty$, we see that
		$$
		\sum_{n=1}^{\infty}n^{-1}\vv\left\{\max_{1\le k\le n}|S_k|>\varepsilon n^{1/p}\right\}<\infty, \mbox{   for  all $\varepsilon>0$.}
		$$
		Therefore,
		$$
		\begin{aligned}
			\infty&>\sum_{n=1}^{\infty}n^{-1}\vv\left\{\max_{1\le k\le n}|S_k|>\varepsilon n^{1/p}\right\}\\
			&=\sum_{k=1}^{\infty}\sum_{n=2^{k-1}}^{2^k-1}n^{-1}\vv\left\{\max_{1\le k\le n}|S_k|>\varepsilon n^{1/p}\right\}\\
			&\ge \frac12\vv\left\{\max_{1\le m\le 2^{k-1}}|S_m|>\varepsilon 2^{k/p}\right\}.
		\end{aligned}
		$$
		By Borel-Cantell lemma under sub-linear expectations (cf. \cite{chen2013strong} or Lemma  1 of \cite{zhang2018marcinkiewicz}), we see that
		$$
		2^{-k/p}\max_{1\le m\le 2^k}|S_m|\rightarrow 0,  \mbox{    a. s.   $\vv$,}
		$$
		which results in $S_n/n^{1/p}\rightarrow 0$, a. s. $\vv$.
	\end{proof}
	
	\bibliographystyle{elsarticle-harv}
	\bibliography{References01}

\begin{thebibliography}{31}
\expandafter\ifx\csname natexlab\endcsname\relax\def\natexlab#1{#1}\fi
\expandafter\ifx\csname url\endcsname\relax
  \def\url#1{\texttt{#1}}\fi
\expandafter\ifx\csname urlprefix\endcsname\relax\def\urlprefix{URL }\fi

\bibitem[{Chen et~al.(2009)Chen, Hu, and Volodin}]{chen2009limiting}
Chen, P., Hu, T.-c., Volodin, A., 2009. Limiting behaviour of moving average
  processes under $\varphi$-mixing aasumption. Statistics \& Probability
  Letters 79, 105--111.

\bibitem[{Chen and Wu(2022)}]{chen2022complete}
Chen, X., Wu, Q., 2022. Complete convergence and complete integral convergence
  of partial sums for moving average process under sub-linear expectations.
  AIMS Mathematics 7~(6), 9694--9715.

\bibitem[{Chen(2016)}]{chen2016strong}
Chen, Z., 2016. Strong laws of large numbers for sub-linear expectations.
  Science China Mathematics 59~(5), 945--954.

\bibitem[{Chen et~al.(2013)Chen, Wu, and Li}]{chen2013strong}
Chen, Z., Wu, P., Li, B., 2013. A strong law of large numbers for non-additive
  probabilities. International Journal of Approximate Reasoning 54~(3),
  365--377.

\bibitem[{Gao and Xu(2011)}]{gao2011large}
Gao, F., Xu, M., 2011. Large deviations and moderate deviations for independent
  random variables under sublinear expectations. Science China Mathematics (in
  Chinese) 41~(4), 337--352.

\bibitem[{Hosseini and Nezakati(2019)}]{hosseini2019complete}
Hosseini, S.~M., Nezakati, A., 2019. Complete moment convergence for the
  dependent linear processes with random coefficients. Acta Mathematica Sinica,
  English Series 35~(8), 1321--1333.

\bibitem[{Hu et~al.(2014)Hu, Chen, and Zhang}]{hu2014big}
Hu, F., Chen, Z., Zhang, D., 2014. How big are the increments of g-brownian
  motion? Science China Mathematics 57~(8), 1687--1700.

\bibitem[{Hu and Yang(2017)}]{hu2017some}
Hu, Z.-C., Yang, Y.-Z., 2017. Some inequalities and limit theorems under
  sublinear expectations. Acta Mathematicae Applicatae Sinica, English Series
  33~(2), 451--462.

\bibitem[{Ko(2015)}]{ko2015complete}
Ko, M.-H., 2015. Complete moment convergence of moving average process
  generated by a class of random variables. Journal of Inequalities and
  Applications 2015~(1), 1--9.

\bibitem[{Kuczmaszewska(2020)}]{kuczmaszewska2020complete}
Kuczmaszewska, A., 2020. Complete convergence for widely acceptable random
  variables under the sublinear expectations. Journal of Mathematical Analysis
  and Applications 484~(1), 123662.

\bibitem[{Meng et~al.(2021)Meng, Wang, and Wu}]{meng2021convergence}
Meng, B., Wang, D., Wu, Q., 2021. Convergence of asymptotically almost
  negatively associated random variables with random coefficients.
  Communications in Statistics-Theory and Methods, 1--15.

\bibitem[{Meng et~al.(2022)Meng, Wang, and Wu}]{meng2022complete}
Meng, B., Wang, D., Wu, Q., 2022. Complete convergence and complete moment
  convergence for weighted sums of extended negatively dependent random
  variables. Communications in Statistics-Theory and Methods 51~(12),
  3847--3863.

\bibitem[{Peng(2007)}]{peng2007g}
Peng, S., 2007. G-expectation, g-brownian motion and related stochastic
  calculus of it{\^o} type. In: Stochastic analysis and applications. Vol.~2.
  Springer, USA, pp. 541--567.

\bibitem[{Peng(2010)}]{peng2010nonlinear}
Peng, S., 2010. Nonlinear expectations and stochastic calculus under
  uncertainty. arXiv preprint arXiv:1002.4546 24.

\bibitem[{Peng(2019)}]{peng2019nonlinear}
Peng, S., 2019. Nonlinear expectations and stochastic calculus under
  uncertainty: with robust CLT and G-Brownian motion. Vol.~95. Springer Nature,
  Berlin, Germany.

\bibitem[{Wu(2020)}]{wu2020precise}
Wu, Q., 2020. Precise asymptotics for complete integral convergence under
  sublinear expectations. Mathematical Problems in Engineering 2020.

\bibitem[{Wu and Jiang(2018)}]{wu2018strong}
Wu, Q., Jiang, Y., 2018. Strong law of large numbers and chover's law of the
  iterated logarithm under sub-linear expectations. Journal of Mathematical
  Analysis and Applications 460~(1), 252--270.

\bibitem[{Xu and Zhang(2019)}]{xu2019three}
Xu, J.~P., Zhang, L.~X., 2019. Three series theorem for independent random
  variables under sub-linear expectations with applications. Acta Mathematica
  Sinica, English Series 35~(2), 172--184.

\bibitem[{Xu and Zhang(2020)}]{xu2020law}
Xu, J.-p., Zhang, L.-x., 2020. The law of logarithm for arrays of random
  variables under sub-linear expectations. Acta Mathematicae Applicatae Sinica,
  English Series 36~(3), 670--688.

\bibitem[{Xu and Cheng(2021{\natexlab{a}})}]{xu2021convergence}
Xu, M., Cheng, K., 2021{\natexlab{a}}. Convergence for sums of iid random
  variables under sublinear expectations. Journal of Inequalities and
  Applications 2021~(1), 1--14.

\bibitem[{Xu and Cheng(2021{\natexlab{b}})}]{xu2021equivalent}
Xu, M., Cheng, K., 2021{\natexlab{b}}. Equivalent conditions of complete th
  moment convergence for weighted sums of iid random variables under sublinear
  expectations. Discrete Dynamics in Nature and Society 2021.

\bibitem[{Xu and Cheng(2021{\natexlab{c}})}]{xu2021precise}
Xu, M., Cheng, K., 2021{\natexlab{c}}. Precise asymptotics in the law of the
  iterated logarithm under sublinear expectations. Mathematical Problems in
  Engineering 2021.

\bibitem[{Xu and Cheng(2022{\natexlab{a}})}]{xu2022small}
Xu, M., Cheng, K., 2022{\natexlab{a}}. How small are the increments of
  g-brownian motion. Statistics \& Probability Letters 186, 1--9.

\bibitem[{Xu and Cheng(2022{\natexlab{b}})}]{xu2022note}
Xu, M., Cheng, K., 2022{\natexlab{b}}. Note on precise asymptotics in the law
  of the iterated logarithm under sublinear expectations. Mathematical Problems
  in Engineering 2022.

\bibitem[{Zhang(2016{\natexlab{a}})}]{zhang2016exponential}
Zhang, L., 2016{\natexlab{a}}. Exponential inequalities under the sub-linear
  expectations with applications to laws of the iterated logarithm. Science
  China Mathematics 59~(12), 2503--2526.

\bibitem[{Zhang(2016{\natexlab{b}})}]{zhang2016rosenthal}
Zhang, L., 2016{\natexlab{b}}. Rosenthal¡¯s inequalities for independent and
  negatively dependent random variables under sub-linear expectations with
  applications. Science China Mathematics 59~(4), 751--768.

\bibitem[{Zhang and Lin(2018)}]{zhang2018marcinkiewicz}
Zhang, L., Lin, J., 2018. Marcinkiewicz¡¯s strong law of large numbers for
  nonlinear expectations. Statistics \& Probability Letters 137, 269--276.

\bibitem[{Zhang(2015)}]{zhang2015donsker}
Zhang, L.-X., 2015. Donsker's invariance principle under the sub-linear
  expectation with an application to chung's law of the iterated logarithm.
  Communications in Mathematics and Statistics 3~(2), 187--214.

\bibitem[{Zhang(2016{\natexlab{c}})}]{zhang2016strong}
Zhang, L.-X., 2016{\natexlab{c}}. Strong limit theorems for extended
  independent and extended negatively dependent random variables under
  non-linear expectations. arXiv preprint arXiv:1608.00710.

\bibitem[{Zhang and Ding(2017)}]{zhang2017further}
Zhang, Y., Ding, X., 2017. Further research on complete moment convergence for
  moving average process of a class of random variables. Journal of
  Inequalities and Applications 2017~(1), 1--11.

\bibitem[{Zhong and Wu(2017)}]{zhong2017complete}
Zhong, H., Wu, Q., 2017. Complete convergence and complete moment convergence
  for weighted sums of extended negatively dependent random variables under
  sub-linear expectation. Journal of Inequalities and Applications 2017~(1),
  1--14.

\end{thebibliography}
	
\end{document}